\documentclass[12pt]{amsart}

\usepackage{amssymb}
\usepackage{bm}
\usepackage{graphicx}
\usepackage{psfrag}
\usepackage{color}
\usepackage{enumitem}
\usepackage{url}

\usepackage{hyperref}
\definecolor{darkblue}{RGB}{0,0,160}
\hypersetup{colorlinks,citecolor=black,filecolor=black,%
            linkcolor=darkblue,urlcolor=darkblue}

\newcommand{\excise}[1]{}

\newtheorem{thm}{Theorem}[section]
\newtheorem{lemma}[thm]{Lemma}

\newtheorem{cor}[thm]{Corollary}
\newtheorem{prop}[thm]{Proposition}

\newtheorem{question}[thm]{Question}
\newtheorem{prob}[thm]{Problem}

\theoremstyle{definition}
\newtheorem{example}[thm]{Example}
\newtheorem{remark}[thm]{Remark}
\newtheorem{defn}[thm]{Definition}

\numberwithin{equation}{section}

\newcommand{\ring}[1]{\ensuremath{\mathbb{#1}}}

\renewcommand\>{\rangle}
\newcommand\<{\langle}

\newcommand\NN{\ring{N}}
\newcommand\kk{\Bbbk}
\newcommand\mm{{\mathfrak m}}
\newcommand\pp{{\mathfrak p}}
\newcommand\ow{\ol w}
\newcommand\oR{\hspace{.3ex}\ol{\hspace{-.3ex}R\hspace{-.05ex}}\hspace{.05ex}}
\newcommand\oQ{\hspace{.15ex}\ol{\hspace{-.15ex}Q\hspace{-.25ex}}\hspace{.25ex}}
\newcommand\oW{\hspace{.1ex}\overline{\hspace{-.1ex}W\hspace{-.1ex}}\hspace{.1ex}{}}
\newcommand\ttt{\mathbf{t}}

\newcommand\app{\mathord\approx}
\newcommand\nil{\infty}
\newcommand\til{\mathord\sim}

\newcommand\into{\hookrightarrow}

\newcommand\minus{\smallsetminus}
\newcommand\nothing{\varnothing}
\renewcommand\iff{\Leftrightarrow}
\renewcommand\implies{\Rightarrow}

\def\ol#1{{\overline {#1}}}

\DeclareMathOperator\ass{Ass} 
\DeclareMathOperator\soc{soc} 
\DeclareMathOperator\irr{Irr} 

\begin{document}

\mbox{}
\vspace{-4ex}
\title[Irreducible decomposition of binomial ideals]
      {Irreducible decomposition of binomial ideals}
\author{Thomas Kahle}
\address{Faculty of Mathematics\\Otto-von-Guericke Universit\"at\\ Magdeburg, Germany} 
\urladdr{\url{http://www.thomas-kahle.de}}
\author{Ezra Miller}
\address{Mathematics Department\\Duke University\\Durham, NC 27708}
\urladdr{\url{http://www.math.duke.edu/~ezra}}
\author{Christopher O'Neill}
\address{Mathematics Department\\Texas A\&M University\\TX 77843}
\urladdr{\url{http://www.math.tamu.edu/~coneill/}}

\makeatletter
  \@namedef{subjclassname@2010}{\textup{2010} Mathematics Subject Classification}
\makeatother

\date{16 March 2016}

\begin{abstract}
Building on coprincipal mesoprimary decomposition [Kahle and Miller,
2014], we combinatorially construct an irreducible decomposition of
any given binomial ideal.  In a parallel manner, for congruences in
commutative monoids we construct decompositions that are direct
combinatorial analogues of binomial irreducible decompositions, and
for binomial ideals we construct decompositions into ideals that are
as irreducible as possible while remaining binomial.  We provide an
example of a binomial ideal that is not an intersection of binomial
irreducible ideals, thus answering a question of Eisenbud and
Sturmfels [1996].\vspace{-2.4ex}
\end{abstract}
\maketitle

\setcounter{tocdepth}{1}
\tableofcontents

\section{Introduction}\label{s:intro}

An ideal in a commutative ring is irreducible if it is not expressible
as an intersection of two ideals properly containing it.  Irreducible
ideals are primary, and any ideal $I$ in a Noetherian ring is an
intersection of irreducible ideals.  These \emph{irreducible
decompositions} are thus special cases of primary decomposition, but
likewise are hard to compute in general.  If $I$ is a monomial ideal,
however, this task is much easier: any monomial ideal is an
intersection of irreducible ideals that are themselves monomial ideals
(see \cite[Theorem~5.27]{cca} for polynomial rings and
\cite[Theorem~2.4]{irredres} for affine semigroup rings), and these
\emph{monomial irreducible decompositions} are heavily governed by
combinatorics.  The ease of monomial irreducible decomposition plus
the existence of binomial primary decomposition in polynomial rings over
algebraically closed fields \cite[Theorem~7.1]{ES96} motivated
Eisenbud and Sturmfels to ask:

\begin{question}[{\cite[Problem~7.5]{ES96}}]\label{q:binomirrdecomp}
Does every binomial ideal over an algebraically closed field admit a
binomial irreducible decomposition?
\end{question}

We answer Question~\ref{q:binomirrdecomp} using the theory of
mesoprimary decomposition \cite{mesoprimary}.  Our response has three
stages.  First, congruences in Noetherian commutative monoids admit
\emph{soccular decompositions} (Theorem~\ref{t:socculardecomp}), which
should be considered the direct combinatorial analogues of binomial
irreducible decompositions.  (\emph{Soccular congruences}
(Definition~\ref{d:soccular}) fail to be irreducible for the same
reason that prime congruences do; see the end of
\cite[Section~2]{mesoprimary} for details.)  Second, lifting to
binomial ideals the method of constructing soccular congruences (but
not lifting the construction itself; see
Example~\ref{e:nonsoccularbinoccular}) yields ideals that are, in a
precise sense, as irreducible as possible while remaining binomial
(Definition~\ref{d:binoccular}).  The resulting notion of
\emph{binoccular decomposition} for binomial ideals
(Theorem~\ref{t:binocculardecomp}) proceeds as far as possible toward
irreducible decomposition while remaining confined to the category of
binomial ideals.  Theorem~\ref{t:nonbinomialdecomp} demonstrates, by
example, that the confines of binomiality can prevent reaching all the
way to irreducible decomposition by exhibiting a binomial ideal not
expressible as an intersection of binomial irreducible ideals, thus
solving Eisenbud and Sturmfels' problem in the negative.  That said,
our third and final stage produces irreducible decompositions of
binomial ideals (Corollary~\ref{c:irreducibledecomp}) in a manner that
is as combinatorial as mesoprimary decomposition: each coprincipal
component has a canonical \emph{irreducible closure}
(Definition~\ref{d:irreducibleclosure}) that, while not itself an
irreducible ideal, has a canonical primary decomposition all of whose
components are irreducible (Theorem~\ref{t:irrclosureprimarydecomp}).

All three of the decompositions in this paper---soccular, binoccular,
and irreducible---descend directly from coprincipal decomposition
\cite[Theorems~8.4 and~13.3]{mesoprimary} (see Theorems~\ref{t:kmcong}
and~\ref{t:kmessential} for restatements of these results).  This is
true in two senses: (i)~the components in all three types of
decomposition are cogenerated by the same witnesses that cogenerate
the corresponding coprincipal components, and (ii)~the components
themselves are constructed by adding new relations to the
corresponding coprincipal components.  To be more precise, soccular
congruences are constructed by adding relations between all pairs of
\emph{protected witnesses} (Definition~\ref{d:protected}) for
coprincipal congruences while maintaining their cogenerators
(Theorem~\ref{t:protected} and Corollary~\ref{c:protected}).
Similarly, binoccular ideals are constructed by repeatedly throwing
into a coprincipal ideal as many binomial socle elements as possible
while maintaining a monomial cogenerator in the socle
(Definitions~\ref{d:binoccularclosure}
and~\ref{d:binoccularcomponent}).  In contrast, irreducible closures
allow arbitrary polynomials to be thrown in, not merely binomials.
Although this concrete description of irreducible closure is accurate,
the construction of irreducible closures
(Definition~\ref{d:irreducibleclosure}) is accomplished with more
abstract, general commutative algebra.  Consequently, the reason why
irreducible closures have canonical irreducible decompositions is
particularly general, from the standpoint of commutative algebra,
involving embeddings of rings inside of Gorenstein localizations
(Remark~\ref{r:Gor}).

Finally, it bears mentioning that for the proofs of correctness---at
least for the decompositions in rings as opposed to monoids---we make
explicit a unifying principle, in the form of equivalent criteria
involving socles and monomial localization
(Lemma~\ref{l:intersection}), for when a binomial ideal in a monoid
algebra equals a given intersection of ideals.

\subsubsection*{Note on prerequisites}
Although the developments here are based on those in
\cite{mesoprimary}, the reader is not assumed to have assimilated the
results there.  The exposition here assumes familiarity only with the
most basic monoid theory used in \cite{mesoprimary}.  To make this
paper self-contained, every result from \cite{mesoprimary} that is
applied here is stated precisely in Section~\ref{s:preliminaries} with
prerequisite definitions.  In fact, Section~\ref{s:preliminaries}
serves as a handy summary of~\cite{mesoprimary}, proceeding through
most of its logical content as efficiently as~possible.

\section{Preliminaries}\label{s:preliminaries}

We need to briefly review some definitions and results from
\cite{mesoprimary}.  Following that paper, we first deal with monoid
congruences (the combinatorial setting) and then the respective
binomial ideal counterparts (the arithmetic setting).  Throughout, let
$Q$ denote a commutative Noetherian monoid and $\kk$ a field.  We
assume familiarity with basic notions from monoid theory; see
Sections~2 and~3 of~\cite{mesoprimary}, which contain an introduction
to the salient points with binomial algebra in mind.  For an example
of the kinds of concepts we assume, an element $q \in Q$ is
\emph{partly cancellative} if $q + a = q + b \neq \nil \implies a = b$
for all cancellative $a,b \in Q$, where $\nil \in Q$ is nil
\cite[Definition~2.9]{mesoprimary}.

\begin{defn}\label{d:cong}
An equivalence relation $\til$ on $Q$ is a \emph{congruence} if
$a \sim b$ implies $a + c \sim b + c$ for all $a, b, c \in Q$.  A
\emph{binomial} in $\kk[Q]$ is an element of the form
$\ttt^a - \lambda\ttt^b$ where $a, b \in Q$ and $\lambda \in \kk$.  An
ideal $I \subset \kk[Q]$ is \emph{binomial} (resp.~\emph{monomial}) if
it can be generated by binomials (resp.~monomials).
\end{defn}

\begin{remark}\label{r:chain}
A binomial ideal $I \subset \kk[Q]$ induces a congruence~$\til_I$ on
$Q$ that sets $a \sim_I b$ whenever $\ttt^a - \lambda\ttt^b \in I$ for
some nonzero $\lambda \in \kk$.  The quotient algebra $\kk[Q]/I$ is
finely graded by the quotient monoid $Q/\til_I$.  Conversely, each
congruence on $Q$ is of the form $\til_I$ for some binomial ideal
$I \subset \kk[Q]$, although more than one~$I$ is possible: the nil
class can be zero or not \cite[Proposition~9.5]{mesoprimary}, and the
congruence~forgets~coefficients.
\end{remark}

\begin{defn}[{\cite[Definitions~2.12, 3.4, 4.7, 4.10, 7.1, 7.2, 7.7,
and~7.12]{mesoprimary}}]\label{d:kmcong}
Fix a congruence~$\til$ on~$Q$ and a prime $P \subset Q$.  Write
$\oQ_P = Q_P/\til$, where $Q_P$ is the localization along~$P$, and
denote by~$\ol q$ the image of~$q \in Q$ in $\oQ = Q/\til$.
\begin{enumerate}
\item%
An element $q \in Q$ is an \emph{aide} for $w \in Q$ and a generator
$p \in P$ if $\ow \ne \ol q$, and $\ow + \ol p = \ol q + \ol p$, and
$\ol q$ is maximal in the set $\{\ol q,\ow\}$.  The element $q$ is a
\emph{key aide} for $w$ if $q$ is an aide for $w$ for each generator
of~$P$.  An element $w \in Q$ is a \emph{witness} for~$P$ if it has an
aide for each $p \in P$, and a \emph{key witness} for~$P$ if it has a
key aide.  A key witness $w$ is a \emph{cogenerator} of $\til$ if
$w + p$ is nil modulo $\til$ for all $p \in P$.

\item%
The congruence~$\til$ is \emph{$P$-primary} if every $p \in P$ is
nilpotent in~$\oQ$ and every $f \in Q \minus P$ is cancellative
in~$\oQ$.  A $P$-primary congruence~$\til$ is \emph{mesoprimary} if
every element of the quotient~$\oQ$ is partly cancellative.  The
congruence~$\til$ is \emph{coprincipal} if it is mesoprimary and every
cogenerator for $\til$ generates the~same~ideal in~$\oQ$.

\item%
The \emph{coprincipal component $\til_w^P$ of~$\til$ cogenerated by a
witness $w \in Q$ for~$P$} is the coprincipal congruence that relates
$a, b \in Q$ if one of the following is satisfied:
\begin{enumerate}
\item%
both $\ol a$ and $\ol b$ generate an ideal not containing $\ol q$ in
$\oQ_P$; or
\item%
$\ol a$ and $\ol b$ differ by a unit in $\oQ_P$ and $\ol a + \ol c =
\ol b + \ol c = \ol q$ for some $\ol c \in \oQ_P$.
\end{enumerate}
\end{enumerate}
A (key) witness for~$P$ may be called a (key) $\til$-witness for~$P$
to specify~$\til$.  Congruences may be called $P$-mesoprimary or
$P$-coprincipal to specify~$P$.
\end{defn}

\begin{defn}[{\cite[Definitions~5.1 and~5.2]{mesoprimary}}]\label{d:primecong}
Fix a congruence~$\til$ on a monoid~$Q$, a prime ideal $P \subset Q$,
and an element $q \in Q$ that is not nil modulo~$\til$.
\begin{enumerate}
\item%
Let $G_P \subset Q_P$ denote the unit group of the localization~$Q_P$,
and write $K_q^P \subset G_P$ for the stabilizer of~$\ol q \in \oQ_P$
under the action of~$G_P$.

\item%
If $\app$ is the congruence on~$Q_P$ that sets $a \approx b$ whenever
\begin{itemize}
\item%
$a$ and $b$ lie in $P_P$ or
\item%
$a$ and $b$ lie in $G_P$ and $a - b \in K_q^P$,
\end{itemize}
then the \emph{$P$-prime congruence} of~$\til$ at $q$ is $\ker(Q \to
Q_P/\app)$.

\item%
The $P$-prime congruence at $q$ is \emph{associated to $\til$} if $q$
is a key witness for~$P$.
\end{enumerate}
\end{defn}

\begin{defn}[{\cite[Definition~8.1]{mesoprimary}}]\label{d:kmmesodecomp}
An expression of a congruence~$\til$ on~$Q$ as a common refinement
$\bigcap_i \app_i$ of mesoprimary congruences is a \emph{mesoprimary
decomposition} of~$\til$ if, for each~$\app_i$ with associated prime
$P_i \subset Q$, the $P_i$-prime congruences of~$\til$ and~$\app_i$ at
each cogenerator for~$\app_i$ coincide.  This decomposition is
\emph{key} if every cogenerator for every~$\app_i$ is a key witness
for~$\til$.
\end{defn}

\begin{thm}[{\cite[Theorem~8.4]{mesoprimary}}]\label{t:kmcong}
Each congruence~$\til$ on $Q$ is the common refine\-ment of the
coprincipal components cogenerated by its key witnesses.
\end{thm}

The proof of Theorem~\ref{t:kmcong} at the source
\cite[Theorem~8.4]{mesoprimary} yields the following corollary, which
is necessary for Theorem~\ref{t:oneshotsocc}.

\begin{cor}\label{c:kmcong}
Given a congruence~$\til$ on $Q$ and elements $a, b \in Q$ with $a
\not\sim b$, there exists a monoid prime $P \subset Q$ and an element
$u \in Q$ such that (after possibly swapping $a$ and $b$) the element
$a + u$ is a key $\til$-witness for~$P$ with key aide $b + u$.
\end{cor}

A few more definitions are required before a precise statement of the
main existence result for binomial ideals from~\cite{mesoprimary} can
be made in Theorem~\ref{t:kmessential}.

\begin{defn}[{\cite[Definitions~11.7, 11.11, and~12.1]{mesoprimary}}]%
\label{d:mesoprime}
Let $I \subset \kk[Q]$ be a binomial ideal.  Fix a prime $P \subset Q$
and an element $q \in Q$ with $\ttt^q \notin I_P$.
\begin{enumerate}
\item%
Let $G_P \subset Q_P$ denote the unit group of~$Q_P$, and write $K_q^P
\subset G_P$ for the subgroup of~$G_P$ that fixes the class of~$q$
modulo $\til_I$.

\item%
Denote by $\rho: K_q^P \to \kk^*$ the group homomorphism such that
$\ttt^u - \rho(u)$ lies in the kernel of the $\kk[G_P]$-module
homomorphism $\kk[G_P] \to \kk[Q_P]/I_P$ taking $1 \mapsto \ttt^q$.

\item%
The \emph{$P$-mesoprime ideal of~$I$ at~$q$} is the preimage $I_q^P$
in $\kk[Q]$ of $(I_q^P)_P = I_\rho + \mm_P$, where
$I_\rho = \<\ttt^u - \rho(u-v)\ttt^v \mid u - v \in K_q^P\> \subset
\kk[Q]_P$.

\item%
An element $w \in Q$ is an \emph{$I$-witness} for a monoid prime~$P$
if $w$ is a $\til_I$-witness for~$P$ or if $P = \nothing$ is empty and
$I$ contains no monomials.  $w \in Q$ is an \emph{essential
$I$-witness} if $w$ is a key $\til_I$-witness or some polynomial in
$\kk[Q_P]/I_P$ annihilated by~$\mm_P$ has $\ttt^w$ minimal (under
Green's preorder) among its nonzero~monomials.

\item%
The mesoprime $I_q^P$ is \emph{associated to $I$} if $q$ is an
essential $I$-witness for~$P$.
\end{enumerate}
\end{defn}

\begin{defn}[{\cite[Definitions~10.4, 12.14, 12.18]{mesoprimary}}]\label{d:kmessential}
Fix a binomial ideal $I \subset \kk[Q]$ and a prime $P \subset Q$.
\begin{enumerate}

\item%
The ideal $I$ is \emph{mesoprimary} (resp.~\emph{coprincipal}) if the
congruence~$\til_I$ is mesoprimary (resp.~coprincipal) and $I$ is
maximal among binomial ideals in $\kk[Q]$ inducing this congruence.

\item%
The \emph{$P$-coprincipal component of~$I$ at $w \in Q$} is the
preimage $W_w^P(I) \subset \kk[Q]$ of the ideal $I_P + I_\rho +
M_w^P(I) \subset \kk[Q]_P$, where $M_w^P(I)$ is the ideal generated by
the monomials $\ttt^u \in \kk[Q]$ such that~\mbox{$w \notin \<u\>
\subset Q_P$}.
\end{enumerate}
\end{defn}

\begin{defn}[{\cite[Definition~13.1]{mesoprimary}}]\label{d:kmcoprincdecomp}
An expression $I = \bigcap_j I_j$ is a \emph{mesoprimary
decomposition} if each component $I_j$ is $P_j$-mesoprimary and the
$P_j$-mesoprimes of~$I$ and~$I_j$ at each cogenerator of~$I_j$
coincide.  This decomposition is \emph{combinatorial} if every
cogenerator of every component is an essential $I$-witness.
A~mesoprimary decomposition is a \emph{coprincipal decomposition} if
every component is coprincipal.
\end{defn}

\begin{thm}[{\cite[Theorem~13.3]{mesoprimary}}]\label{t:kmessential}
Every binomial ideal $I \subset \kk[Q]$ is the intersection of the
coprincipal components cogenerated by its essential witnesses.  In
particular, every binomial ideal admits a combinatorial mesoprimary
decomposition.
\end{thm}

Theorem~\ref{t:kmessential} produces a primary decomposition of any
binomial ideal via the next result.  Precise details about the primary
components here can be found at the cited locations
in~\cite{mesoprimary}.

\begin{prop}[{\cite[Corollary~15.2 and Proposition~15.4]{mesoprimary}}]%
\label{p:kmprimarydecomp}
Fix a mesoprimary ideal~$I \subset \kk[Q]$.  The associated primes
of~$I$ are exactly the minimal primes of its associated mesoprime.
Consequently, $I$ admits a canonical minimal primary decomposition.
When $\kk = \ol\kk$ is algebraically closed, every component of this
decomposition is binomial.
\end{prop}

\begin{thm}[{\cite[Theorems~15.6 and 15.12]{mesoprimary}}]\label{t:kmprimdecomp}
Fix a binomial ideal $I \subseteq \kk[Q]$.  Each associated prime
of~$I$ is minimal over some associated mesoprime of~$I$.  If\/~$\kk =
\ol\kk$ is algebraically closed, then refining any mesoprimary
decomposition of~$I$ by canonical primary decomposition of its
components yields a binomial primary decomposition of~$I$.
\end{thm}

\section{Soccular congruences}\label{s:soccular}

Although the condition to be a coprincipal quotient is strong, it does
not imply that a binomial ideal inducing a coprincipal congruence has
simple socle.  Precisely, the socle of a coprincipal quotient has only
one monomial up to units locally at the associated prime.
While this suffices for irreducible decomposition of monomial ideals,
modulo a binomial ideal the socle can have binomials and general
polynomials.  Our first step is soccular decomposition
(Theorem~\ref{t:oneshotsocc}), which parallels, at the level of
congruences, the construction of irreducible decompositions of
binomial ideals (Theorem~\ref{t:irreducibledecomp}).  While it is the
optimal construction in the combinatorial setting, soccular
decomposition cannot yield irreducible decompositions of binomial
ideals in general since these need not be binomial
(Example~\ref{e:nonbinomialdecomp}).  To start, here is a simple
example of a primary coprincipal binomial ideal that is reducible,
demonstrating that coprincipal decomposition of ideals is not
irreducible~\mbox{decomposition}.

\begin{example}\label{e:congsocle}
The congruence on~$\NN^2$ induced by the ideal $I = \<x^2-xy, xy-y^2,
x^3\>$ is coprincipal, but $x - y \in \soc_\mm(I)$ for
$\mm = \<x,y\>$.  This is because $x$ and $y$ are both key
witnesses and each is an aide for the other.
\end{example}

\begin{defn}\label{d:soccular}
A congruence~$\til$ on $Q$ is \emph{soccular} if its key witnesses all
generate the same principal ideal in the localized
quotient~$Q_P/\til$.
\end{defn}

\begin{defn}\label{d:soccularcollapse}
Fix a monoid prime $P \subset Q$ and a $P$-coprincipal congruence
$\til$ on $Q$ with cogenerator~$w \in Q$.  The \emph{(first) soccular
collapse of~$\til$} is the congruence~$\app$ that sets $a \approx b$
if $a, b \notin \<w\>$ and $a + p \sim b + p$ for all $p \in P$.  The
\emph{$i$-th soccular collapse of~$\til$} is the soccular collapse of
the $(i-1)$st soccular collapse of~$\til$.
\end{defn}

Soccular collapses remove key witness pairs that are not Green's
equivalent to the cogenerator of a coprincipal congruence.  It is
routine to check that the soccular collapse of a coprincipal
congruence is a coprincipal congruence (see the following lemmas).
The construction stabilizes since $Q$ is a Noetherian monoid and
consequently the iterated soccular collapse of a coprincipal
congruence is a soccular congruence.

In general, to form a congruence from a set of relations, one takes
monoid closure and then transitive closure.  Lemma~\ref{l:noclosure}
says that for a soccular collapse of a coprincipal congruence, both of
these operations are trivial.

\begin{lemma}\label{l:noclosure}
The soccular collapse of a $P$-coprincipal congruence~$\til$ 
is a congruence on $Q$ that coarsens $\til$.  
\end{lemma}
\begin{proof}
The soccular collapse $\app$ is symmetric and transitive since $\til$
is symmetric and transitive.  Suppose $a, b \notin \<w\>$ with
$a + p \sim b + p$ for all $p \in P$.  Then for all $q \in Q$,
$a + q + p \sim b + q + p$ for all $p \in P$ since $q + p \in P$, so
$a + q \approx b + q$.  Therefore $\app$ is a congruence on $Q$.
Lastly, if $a \sim b$, then $a + p \sim b + p$ for all $p \in P$, so
$\til$ refines~$\app$.
\end{proof}

\begin{lemma}\label{l:notmax}
Resuming the notation from Definition~\ref{d:soccularcollapse}, if $a
\approx b$ and $a \not\sim b$, then neither $a$ nor $b$ is maximal in
$Q$ modulo Green's relation.
\end{lemma}
\begin{proof}
Given Lemma~\ref{l:noclosure}, the definition of~$\app$ ensures that
$a$ and $b$ both preceed $w$ modulo Green's relation, which ensures
$a$ and $b$ are not maximal.
\end{proof}

Lemma~\ref{l:diffcancel} shows that taking the soccular collapse of a
coprincipal congruence does not modify Green's classes.

\begin{lemma}\label{l:diffcancel}
Resuming the notation from Definition~\ref{d:soccularcollapse}, if $a,
b \in Q$ differ by a cancellative modulo $\til$, then soccular
collapse does not join them.
\end{lemma}
\begin{proof}
Suppose $a \approx b$ and $a = b + f$ for some cancellative element
$f$.  For each $p \in P$, $a + p = b + p$ by Lemma~\ref{l:noclosure},
and each is non nil by Lemma~\ref{l:notmax}.  Thus
$b + f + p \sim b + p$, so $f = 0$ by the partly cancellative property
of $b + p$ in Definition~\ref{d:kmcong}.2.
\end{proof}

\begin{prop}\label{p:coarsening}
Fix a $P$-coprincipal congruence~$\til$ on $Q$ with cogenerator~$w$.
The soccular collapse $\app$ of~$\til$ is coprincipal with
cogenerator~$w$, and $\app$ coarsens $\til$.  Moreover, the elements
$a, b \in Q$ distinct under $\til$ but identified under $\app$ are
precisely the key witnesses of~$\til$ lying outside the Green's class
of~$w$.
\end{prop}
\begin{proof}
The congruence~$\app$ coarsens $\til$ by Lemma~\ref{l:noclosure}.  As
$\til$ is mesoprimary, Lemma~\ref{l:diffcancel} ensures that $\app$ is
also mesoprimary, and by Lemma~\ref{l:notmax} $\app$ agrees with
$\til$ on the Green's class of~$w$.  The final claim follows upon
observing that $a$ and $b$ are by definition key witnesses for $\til$.
\end{proof}

\begin{defn}\label{d:keypair}
Fix a $P$-coprincipal congruence~$\til$ on $Q$.  Two distinct key
witnesses $a, b \in Q$ for $\til$ form a \emph{key witness pair} if
each is a key aide for the other.
\end{defn}

\begin{remark}\label{r:protected}
If $a,b \in Q$ form a key witness pair under a coprincipal
congruence~$\til$ and neither of them is Green's equivalent to the
cogenerator $w$, then they are no longer a key witness pair under the
soccular collapse $\app$ of~$\til$ by Proposition~\ref{p:coarsening}.
However, $\app$ may still have key witnesses, as shown in
Example~\ref{e:double}.
\end{remark}

\begin{example}\label{e:double}
Let
$I = \<x^3 - x^2y, x^2y - xy^2, xy^3 - y^4, x^5\> \subset \kk[x,y]$.
The congruence~$\til_I$ and its soccular collapse are shown in
Figure~\ref{fig:ex302}.  The monoid element $xy$ is a key witness for
$\til_I$, where it is paired with $y^2$, as well as for the soccular
collapse of~$\til_I$, where it is paired with $x^2$.
\end{example}

\begin{figure}[tbp]
\begin{center}
\includegraphics[width=1.5in]{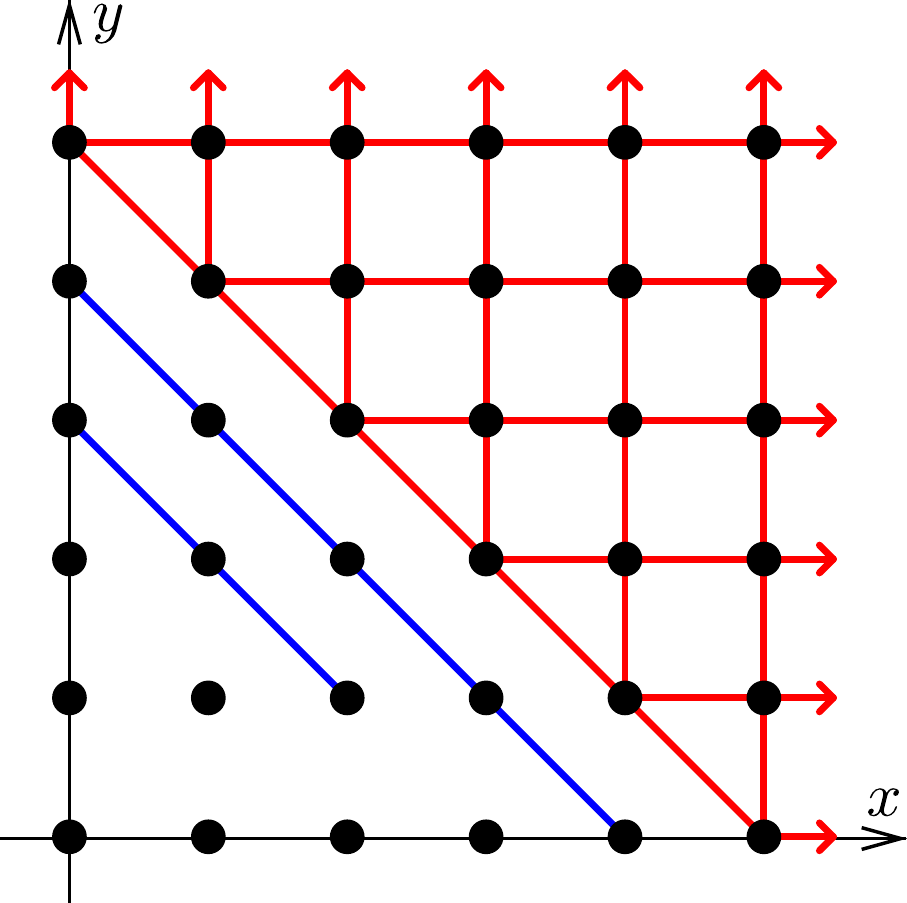}
\hspace{1in}
\includegraphics[width=1.5in]{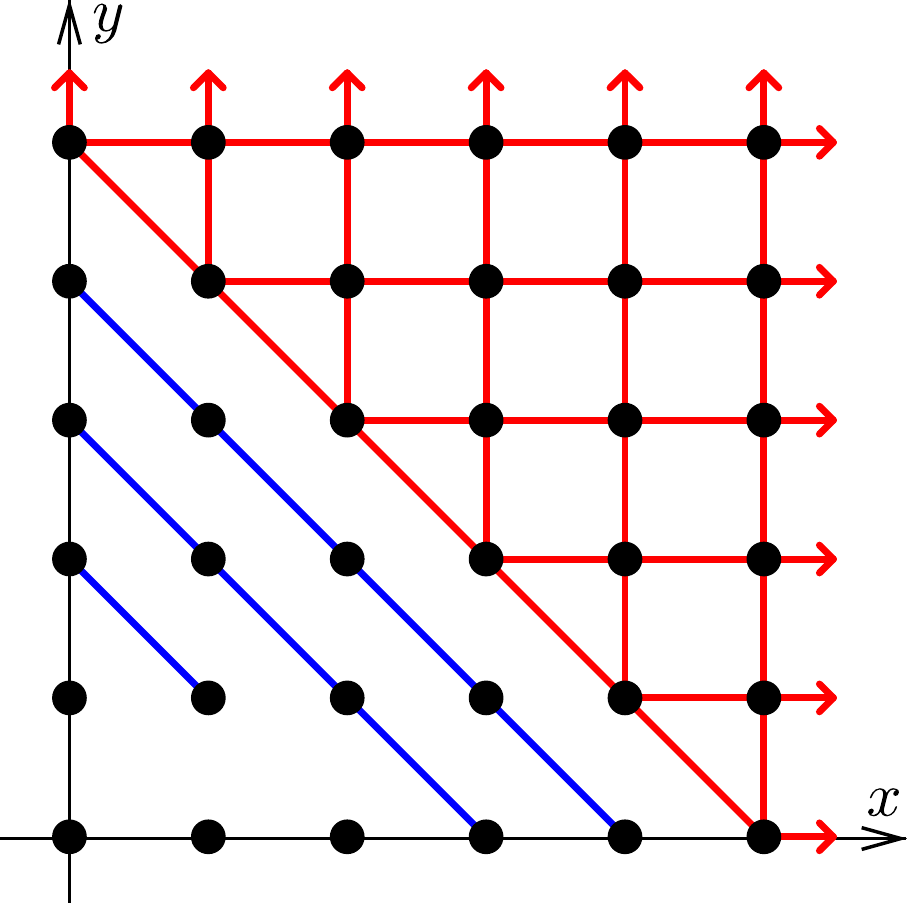}
\end{center}
\caption[A key witness that is also a protected witness]{For $I =
\<x^3 - x^2y, x^2y - xy^2, xy^3 - y^4, x^5\> \subset \kk[x,y]$, the
congruence induced by~$I$ on~$\NN^2$ (left), and its soccular collapse
(right).  The monomial $xy$ is a key witness for both $\til_I$ and
its~\mbox{soccular collapse}.}
\label{fig:ex302}
\end{figure}

\begin{defn}\label{d:protected}
Fix a coprincipal congruence~$\til$ on $Q$ with cogenerator $w$.  An
element $a \in Q$ is a \emph{protected witness for $\til$} if it is a
key witness for the $i$th soccular collapse of~$\til$ for some
$i \ge 1$.  Elements $a, b \in Q$ form a \emph{protected witnesses
pair} if they form a key witness pair for some iterated soccular
collapse of $\til$.
\end{defn}

\begin{defn}\label{d:soccularclosure}
Fix a coprincipal congruence~$\til$ on $Q$.  The \emph{soccular
closure} $\ol\til$ of~$\til$ is the congruence refined by $\til$ that
additionally joins any $a$ and $b$ related under some soccular
collapse of~$\til$.
\end{defn}

\begin{lemma}\label{l:soccularclosure}
Fix a coprincipal congruence~$\til$ on $Q$ with cogenerator $w$.  
The soccular closure $\ol\til$ of~$\til$ is a soccular congruence, 
and its set of key witnesses is exactly the Green's class of~$w$.  
\end{lemma}
\begin{proof}
By construction, the soccular closure has no key witnesses outside
the Green's class of~$w$.
\end{proof}

We now characterize protected witnesses and give a non-iterative way
to specify the soccular closure.  To this end, let
$$%
  (w :_{\til} q) = \{p \in Q \mid \ol q + \ol p = \ow \text{ in }
  Q_P/\til\}.
$$

\begin{thm}\label{t:protected}
Fix a $P$-coprincipal congruence~$\til$ on $Q$ with cogenerator $w$,
and write $\oQ = Q/\til$.  Then $q, q' \in Q$ with distinct classes in
$\oQ$ are a protected witnesses pair for~$\til$ if and only if
$(w :_{\til} q) = (w :_{\til} q')$.
\end{thm}
\begin{proof}
Let $\ol\til$ denote the soccular closure of~$\til$.  Since passing
to~$\ol\til$ leaves the class of~$w$ under~$\til$ unchanged,
$(w :_{\til} q) = (w :_{\ol\til} q)$ for all $q \in Q$.  Therefore, if
$q$ and~$q'$ are merged under~$\ol\til$, the sets $(w :_{\til} q)$ and
$(w :_{\til} q')$ coincide.

Now assume $q$ and~$q'$ are not related under~$\ol\til$.  Pick an
element $p \in P$ such that $q + p$ and $q' + p$ are distinct
under~$\ol\til$ and such that the image $\ol p \in \oQ$ is maximal
among images of elements in~$P$ with this property.  Existence of~$p$
is guaranteed because $\til$ is primary, whence $\oQ_P$ has only
finitely many Green's classes.  Maximality of~$p$ implies that $q + p$
and $q' + p$ become merged in $\oQ/\ol\til$ under the action of any
element of~$P$.  Since $\ol\til$ has no key witness pairs, one
of~$q + p$ and $q' + p$ must be nil, and maximality of~$p$ implies the
other is Green's equivalent to $w$.  After possibly switching $q$ and
$q'$, this gives $p \in (w :_{\til} q)$ but
$p \notin (w :_{\til} q')$.
\end{proof}

\begin{cor}\label{c:protected}
Fix a coprincipal congruence~$\til$ on $Q$ cogenerated by $w$.  The
soccular closure $\ol\til$ of~$\til$ relates $a$ and $b$ if and only
if $(w :_{\til} a) = (w :_{\til} b)$.  \qed
\end{cor}

\section{Soccular decomposition of congruences}\label{s:congdecomp}

Every congruence can be expressed as a common refinement of soccular
congruences.  Our constructive proof first produces the decomposition
in Corollary~\ref{c:socculardecomp}, which might not be a mesoprimary
decomposition; see Remark~\ref{r:soccularmesodecomp}.
Theorem~\ref{t:oneshotsocc} removes unnecessary components and shows
that the resulting decomposition is mesoprimary.

\begin{defn}\label{d:socculardecomp}
Fix a $P$-coprincipal congruence~$\til$ on $Q$ and a key witness $w
\in Q$.  The \emph{soccular component $\ol\til_w^P$ of~$\til$
cogenerated by $w$ along $P$} is the soccular closure of the
coprincipal component $\til_w^P$ cogenerated by $w$ along $P$.
\end{defn}

\begin{thm}\label{t:socculardecomp} 
Any coprincipal congruence~$\til$ on $Q$ is the common refinement 
of the soccular components cogenerated by its protected witnesses.  
\end{thm}
\begin{proof}
Each soccular component coarsens $\til$ by Lemma~\ref{l:noclosure}, so
it suffices to show that their common refinement is $\til$.  Let
$w \in Q$ denote a cogenerator of~$\til$ and fix distinct
$a, b \in Q$.  If the soccular component of~$\til$ at $w$ (that is,
the soccular closure of~$\til$) leaves $a$ and $b$ distinct, we are
done.  Otherwise, both $a$ and $b$ are protected witnesses, and the
soccular component of~$\til$ at $a$ joins $b$ with the nil class.
\end{proof}

\begin{cor}\label{c:socculardecomp} 
Any congruence~$\til$ on $Q$ can be expressed as a common refinement
of soccular congruences.
\end{cor}
\begin{proof}
Apply Theorem~\ref{t:kmcong} to $\til$, then
Theorem~\ref{t:socculardecomp} to each component.
\end{proof}

\begin{remark}\label{r:soccularmesodecomp}
The decomposition in Corollary~\ref{c:socculardecomp} is not
necessarily a mesoprimary decomposition in the sense of
Definition~\ref{d:kmmesodecomp}, since the associated prime congruence
of a component $\app$ cogenerated at a protected witness $q \in Q$
need not coincide with the prime congruence at $q$ under $\til$.  The
next theorem shows that the components in this decomposition
cogenerated at protected witnesses that are not key $\til$-witnesses
are redundant, and the resulting decomposition is indeed a mesoprimary
decomposition.
\end{remark}

\begin{thm}\label{t:oneshotsocc}
Any congruence~$\til$ is the common refinement of the soccular
components cogenerated by its key witnesses.
\end{thm}
\begin{proof}
For elements $a, b \in Q$ with $a \not\sim b$,
Corollary~\ref{c:kmcong} produces, after possibly swapping $a$ and
$b$, a prime $P \subset Q$ and $u \in Q$ such that $a \not\sim_w^P b$
for a key witness $w = a + u$ with key aide $b + u$.  Since
$\ol\til_w^P$ has the same cogenerator and nil class as $\til_w^P$,
Corollary~\ref{c:protected} ensures that $\ol\sim_w^P$ does not relate
$a$ and $b$ as well.
\end{proof}

\section{Binoccular decomposition of binomial ideals}\label{s:idealdecomp}

The binomial ideal analogue (Theorem~\ref{t:binocculardecomp}) of
soccular decomposition (Theorem~\ref{t:oneshotsocc}) yields a
decomposition into binoccular ideals (Definition~\ref{d:binoccular}),
each of whose socles contains a monomial cogenerator and no two-term
binomials other than linear combinations of monomial cogenerators.
Due to the possibility of non-binomials in the socle, binoccular
decomposition is not irreducible decomposition, but it is the best
approximation that does not exit the class of binomial ideals.  As
with coprincipal decomposition, the relevant witnesses are essential
witnesses rather than key witnesses.

For any monoid prime ideal $P\subset Q$, let $\mm_P = \<\ttt^p:p\in
P\>$.  In general, a monoid prime ideal~$P$ in a subscript denotes
monomial localization, which arises from inverting all monomials
outside of~$\mm_P$.  (This notation was used in
\cite[Section~11]{mesoprimary}.)

\begin{defn}\label{d:binoccular}
Fix a binomial ideal $I \subset \kk[Q]$ and a prime monoid ideal $P
\subset Q$.  The \emph{$P$-socle of~$I$} is the ideal
$$%
  \soc_P(I) = \{f \in \kk[Q]_P/I_P \mid \mm_P f = 0 \} \subset
  \kk[Q]_P/I_P.
$$
A binomial ideal $I \subset \kk[Q]$ is \emph{binoccular} if it is
$P$-coprincipal and every monomial appearing in each binomial in
$\soc_P(I)$ is a monomial cogenerator of~$\kk[Q]_P/I_P$.
\end{defn}

\begin{example}\label{e:nonsoccularbinoccular}
Binoccular ideals need not induce soccular congruences.  The ideal
$I = \<x^2 - xy, xy + y^2\>$ is $\<x,y\>$-coprincipal since it
contains all monomials of degree 3.  The monomials $x$ and $y$ form a
key witness pair for $\til_I$, but $I$ is irreducible, so these
monomials do not form a binomial socle element.
\end{example}

Example~\ref{e:nonsoccularbinoccular} implies that the witness
protection program in Section~\ref{s:soccular} cannot be expected to
lift directly to the arithmetic setting, in the sense that collapsing
the congruence $\til_I$ combinatorially need not reflect an operation
on~$I$ itself.  Nonetheless, the analogous arithmetic collapse is
easily defined and has the desired effect.

\begin{defn}\label{d:binoccularclosure}
Fix a $P$-coprincipal binomial ideal $I \subset \kk[Q]$ cogenerated by
$w \in Q$.  The \emph{(first) binoccular collapse of~$I$} is the ideal
$$%
  I_1 = \<\ttt^a - \lambda\ttt^b \mid \ttt^p(\ttt^a -
  \lambda\ttt^b) \in I \text{ for all } p \in P\>
$$
and the \emph{$i$-th binoccular collapse $I_i$ of~$I$} is the
binoccular collapse of~$I_{i-1}$.  The \emph{binoccular closure of
$I$} is the smallest ideal~$\ol I$ containing all binoccular collapses
of~$I$.
\end{defn}

\begin{prop}\label{p:binoccularclosure}
Fix a $P$-coprincipal binomial ideal $I \subset \kk[Q]$ cogenerated by
$w \in Q$.  The binoccular collapse $I_1$ of~$I$ is also a coprincipal
ideal cogenerated by~$w$, and for any binomial $\ttt^a - \lambda\ttt^b
\in I_1$ outside of~$I$, the elements $a$ and~$b$ form a key witness
pair~for~$\til_I$.
\end{prop}
\begin{proof}
This follows from Definition~\ref{d:binoccularclosure} and
Proposition~\ref{p:coarsening} since $\til_J$ coarsens $\til_I$ and
refines~$\ol\til_I$.
\end{proof}

\begin{defn}\label{d:binoccularcomponent}
Fix a binomial ideal $I \subset \kk[Q]$, a prime $P \subset Q$, and $w
\in Q$.  The \emph{binoccular component of~$I$ cogenerated by $w$} is
the binoccular closure $\oW_w^P(I)$ of the coprincipal component
$W_w^P(I)$ of~$I$ cogenerated by $w$ along $P$.
\end{defn}

Lemma~\ref{l:intersection} is the core of the original proof of
Theorem~\ref{t:kmessential}, but it was not stated explicitly in these
terms.  This unifying principle is also important as we construct
binoccular decompositions of binomial ideals
(Theorem~\ref{t:binocculardecomp}) and irreducible decompositions of
binomial ideals (Theorem~\ref{t:irreducibledecomp}).

\begin{lemma}\label{l:intersection}
Fix a binomial ideal $I \subset R = \kk[Q]$ and (not necessarily
binomial) ideals $W_1, \ldots, W_r$ containing~$I$.  The following are
equivalent.
\begin{enumerate}[label=\arabic*.,ref=\arabic*]
\item\label{i:intersect}%
$I = W_1 \cap \cdots \cap W_r$.
\item\label{i:quotInj}%
The natural map $R/I \to R/W_1 \oplus \cdots \oplus R/W_r$ is
injective.
\item\label{i:monLocInj}%
The natural map $\soc_P(I) \to R_P/(W_1)_P \oplus \cdots \oplus
R_P/(W_r)_P$ is injective for every monoid prime $P \subset Q$
associated to $\til_I$.
\item\label{i:ringLocInj}%
The natural map
$\soc_\pp(I) \to R_\pp/(W_1)_\pp \oplus \cdots \oplus R_\pp/(W_r)_\pp$
is injective for every prime $\pp \in \ass(I)$.
\end{enumerate}
\end{lemma}
\begin{proof}
The containments $I \subseteq W_1, \ldots, I \subseteq W_r$ induce a
well-defined homomorphism
$$%
  R/I \to R/W_1 \oplus \cdots \oplus R/W_r.
$$
whose kernel is $W_1 \cap \cdots \cap W_r$ modulo~$I$.  Thus
$I = W_1 \cap \cdots \cap W_r$ holds if and only if this map is
injective and therefore \ref{i:intersect}~$\iff$~\ref{i:quotInj}.
Assume the homomorphism just constructed is injective.  Exactness of
localization produces an injective map
$$%
  R_P/I_P \into R_P/(W_1)_P \oplus \cdots \oplus R_P/(W_r)_P
$$
for each monoid prime $P \subset Q$.  This proves
\ref{i:quotInj}~$\implies$~\ref{i:monLocInj}.
Now assume \ref{i:monLocInj} holds and fix a prime $\pp \in \ass(I)$.
By Theorem~\ref{t:kmprimdecomp}, $\pp$ is minimal over some associated
mesoprime of~$I$.  Since $P$ is associated to $\til_I$, the map
$$%
  \soc_P(I) \to R_P/(W_1)_P \oplus \cdots \oplus R_P/(W_r)_P
$$
is injective.  Every monomial outside of~$\mm_P$ also lies
outside of~$\pp$, so by inverting the remaining elements
outside of~$\pp$, we obtain the injection
$$%
  \soc_P(I)_\pp \to R_\pp/(W_1)_\pp \oplus
  \cdots \oplus R_\pp/(W_r)_\pp.
$$
Any element in $\soc_P(I)_\pp$ is annihilated by $\mm_P$, so
$\soc_\pp(I) \subset \soc_P(I)_\pp$, yielding
\ref{i:monLocInj}~$\implies$~\ref{i:ringLocInj}.
Finally, suppose \ref{i:ringLocInj} holds.  Fix a nonzero $f \in R/I$
and a prime $\pp$ minimal over the annihilator of~$f$.  The image
$\ol f \in R_\pp/I_\pp$ of~$f$ is nonzero since $\pp$ contains the
annihilator of~$f$.  Minimality of~$\pp$ implies some power of~$\pp$
annihilates $\ol f$, so $a \ol f$ is annihilated by $\pp$ for some
$a \in \pp$.  By assumption, $a \ol f$ has nonzero image in some
$(R/W_i)_\pp$, meaning $af$ has nonzero image in $R/W_i$.  This proves
\ref{i:ringLocInj}~$\implies$~\ref{i:quotInj}.
\end{proof}

\begin{thm}\label{t:binocculardecomp}
For any binomial ideal $I \subset \kk[Q]$, the intersection of the
binoccular components cogenerated by its essential $I$-witnesses is a
mesoprimary decomposition~of~$I$.
\end{thm}
\begin{proof}
Fix a monoid prime $P \subset Q$ associated to $\til_I$ and a nonzero
$f \in \soc_P(I)$.  By Lemma~\ref{l:intersection}, it suffices to show
that $f$ is nonzero modulo the localization along $P$ of some
binoccular component.  By Definition~\ref{d:mesoprime}.4, some nonzero
monomial $\lambda \ttt^w$ of~$f$ is an essential $I_P$-witness
for~$P$.  This means every monomial of~$f$ other than $\lambda \ttt^w$
that is nonzero modulo $W_w^P(I)_P$ is Green's equivalent to $w$, so
$f$ has nonzero image in the binoccular closure $\oW_w^P(I)_P$.
\end{proof}

\section{Nonexistence of binomial irreducible decomposition}\label{s:nonexistence}

The only binomials in the socle of a binoccular binomial ideal are
binomials where both terms are monomial cogenerators.  When the
monomial ideal~$\mm_P$ for the associated monoid prime~$P$ is a
maximal ideal in~$\kk[Q]$, this means that in fact the socle has
exactly one binomial, up to scale, namely the unique monomial
cogenerator.  However, even in that case the socle can contain
non-binomial elements, too.

\begin{example}\label{e:nonbinomialdecomp}
Let $I = \<x^2y - xy^2, x^3, y^3\> \subset \kk[x,y]$.  This ideal is
binoccular, and its congruence is depicted in Figure~\ref{fig:ex303}.
The binomial generator forces $x^2y^2 \in I$, so $I$ is cogenerated by
$x^2y$.  The monomials $x^2$, $xy$ and $y^2$ are all non-key
witnesses, and $x^2 + y^2 - xy \in \soc_P(I)$ for $\mm_P = \<x,y\>$.
The expression $I = \<x^2 + y^2 - xy, x^3, y^3\> \cap \<x^3,y\>$ is an
irreducible decomposition of~$I$, and as we shall see in
Theorem~\ref{t:nonbinomialdecomp}, every irreducible decomposition of
$I$ contains some non-binomial irreducible component.
\end{example}

\begin{figure}[tbp]
\begin{center}
\includegraphics[width=1.5in]{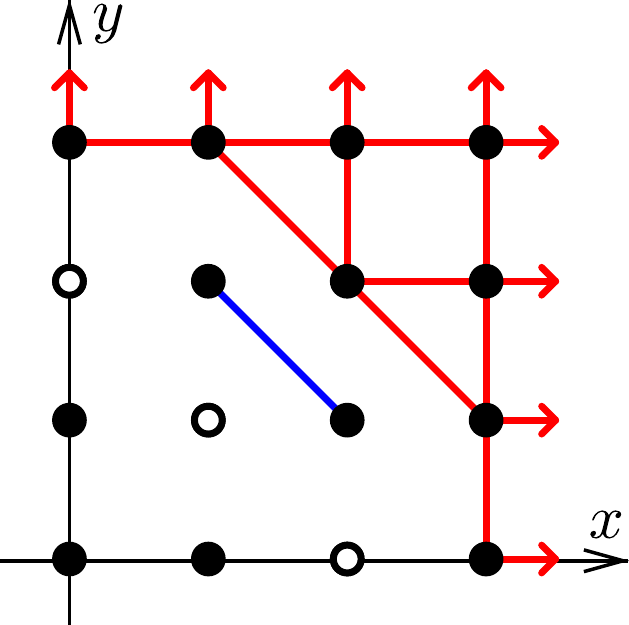}
\end{center}
\caption[A binomial ideal that has no binomial irreducible
decomposition]{The congruence induced by $\<x^2y - xy^2, x^3, y^3\>
\subset \kk[x,y]$ on~$\NN^2$.  The non-binomial element $x^2 + y^2 -
xy$ lies in the socle of~$I$, and as such, $I$ does not admit a
binomial irreducible decomposition.}
\label{fig:ex303}
\end{figure}

Theorem~\ref{t:nonbinomialdecomp} shows that the ideal in
Example~\ref{e:nonbinomialdecomp} cannot be written as the
intersection of binomial irreducible ideals, answering
Question~\ref{q:binomirrdecomp} in the negative.  Its proof uses an
alternative characterization of irreducible ideals in terms of their
socles (Lemma~\ref{l:vasconcelos}).

\begin{defn}\label{d:socle}
Fix an ideal $I$ in a Noetherian ring $R$ and a prime ideal $\pp
\subset R$.  The \emph{$\pp$-socle of~$I$} is
$$%
  \soc_\pp(I) = \{f \in R_\pp/I_\pp \mid
  \pp f = 0\} \subseteq R_\pp/I_\pp.
$$
$I$ has \emph{simple} socle if $\dim_{\kk(\pp)}(\soc_\pp(I)) = 1$,
where $\kk(\pp) = R_\pp/\pp_\pp$ is the residue~field~at~$\pp$.
\end{defn}

\begin{lemma}[{\cite[Proposition~3.1.7]{Vas}}]\label{l:vasconcelos}
The number of components in any irredundant irreducible decomposition
of a $\pp$-primary ideal $I$ in a Noetherian ring $R$ equals
$\dim_{\kk(\pp)}\soc_\pp(I)$.
\end{lemma}

\begin{thm}\label{t:nonbinomialdecomp}
The ideal $I = \<x^2y - xy^2, x^3, y^3\> \subset \kk[x,y]$ cannot be
expressed as an intersection of binomial irreducible ideals.
\end{thm}
\begin{proof}
Let $\mm_P = \<x,y\>$.  The $\kk$-vector space $\soc_P(I)$ is spanned
by $\alpha = x^2 + y^2 - xy$ and $\beta = x^2y$.  Since
$\dim_\kk(\soc_P(I)) = 2$ and $\kk = \kk(\mm_P)$, any irredundant
irreducible decomposition of~$I$ has exactly 2 components by
Lemma~\ref{l:vasconcelos}.  Suppose $I = I_1 \cap I_2$ with $I_1$
and~$I_2$ irreducible.  The equivalence of parts~\ref{i:intersect}
and~\ref{i:quotInj} in Lemma~\ref{l:intersection} implies that the
natural map $\kk[x,y]/I \to \kk[x,y]/I_1 \oplus \kk[x,y]/I_2$ induces
an injection $\soc_{\mm_P}(I) \into \soc_{\mm_P}(I_1) \oplus
\soc_{\mm_P}(I_2)$ which is an isomorphism for dimension reasons.
Possibly exchanging $I_1$ and~$I_2$, assume $f = \alpha +
\lambda\beta$ spans $\soc_{\mm_P}(I_1)$ for some $\lambda \in \kk$.
This implies $f \in I_2$ and $\soc_{\mm_P}(I + \<f\>) = \soc_{\mm_P}
(I_2)$, the latter by an explicit, elementary calculation.
Lemma~\ref{l:intersection} yields $I_2 = I + \<f\>$.
\end{proof}

Example~\ref{e:nonbinomialdecomp} is the first example of a binomial
ideal that does not admit a binomial irreducible decomposition.
However, it is still possible to construct a (not necessarily
binomial) irreducible decomposition from essentially combinatorial
data, as Corollary~\ref{c:irreducibledecomp} demonstrates.

Example~\ref{e:binomialdecompproblem} exhibits the difficulties in
determining whether or not a given binomial ideal admits a binomial
irreducible decomposition.  This question is closely connected with
understanding which components in a coprincipal decomposition are
redundant.

\begin{example}\label{e:binomialdecompproblem}
Consider the two ideals $I = \<x^2y - xy^2, x^4 - x^3y, xy^3 - y^4,
x^5\>$ and $J = \<x^4y - x^3y^2, x^2y^3 - xy^4, x^6 - x^5y, xy^5 -
y^6, x^7\>$, whose respective congruences are depicted in
Figure~\ref{fig:ex304}.
\begin{figure}[tbp]
$$%
  \begin{array}{ccc}
  \includegraphics[width=1.5in]{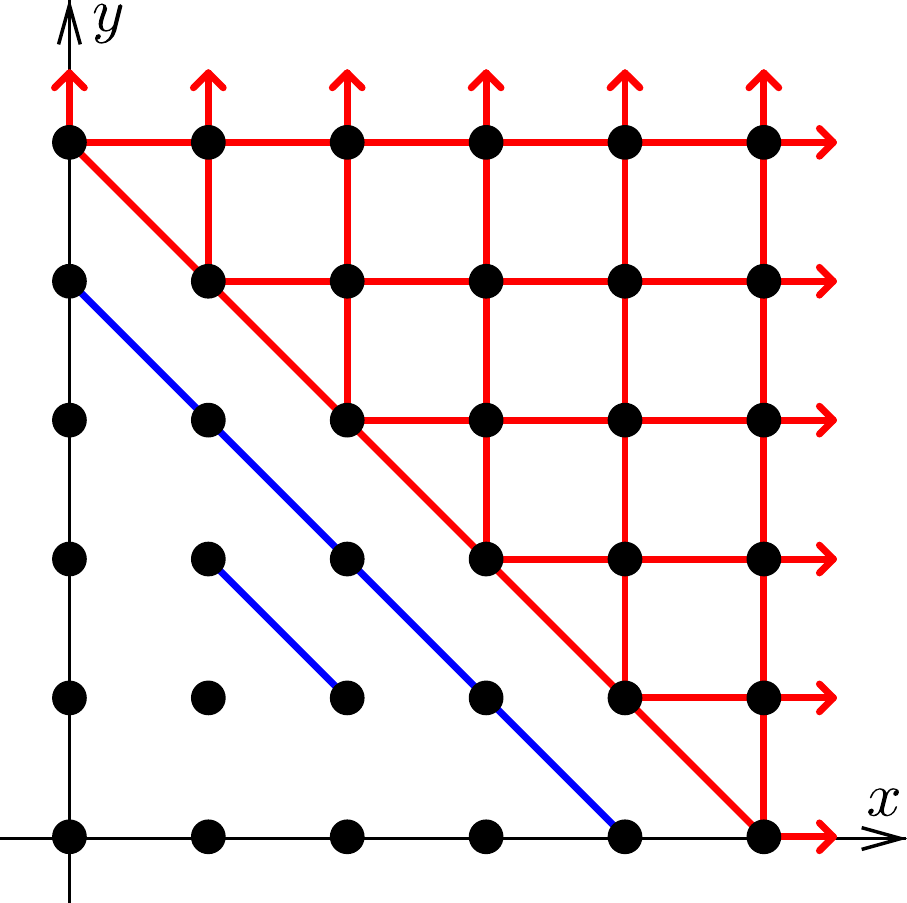}
  & \hspace{3em} &
  \includegraphics[width=1.5in]{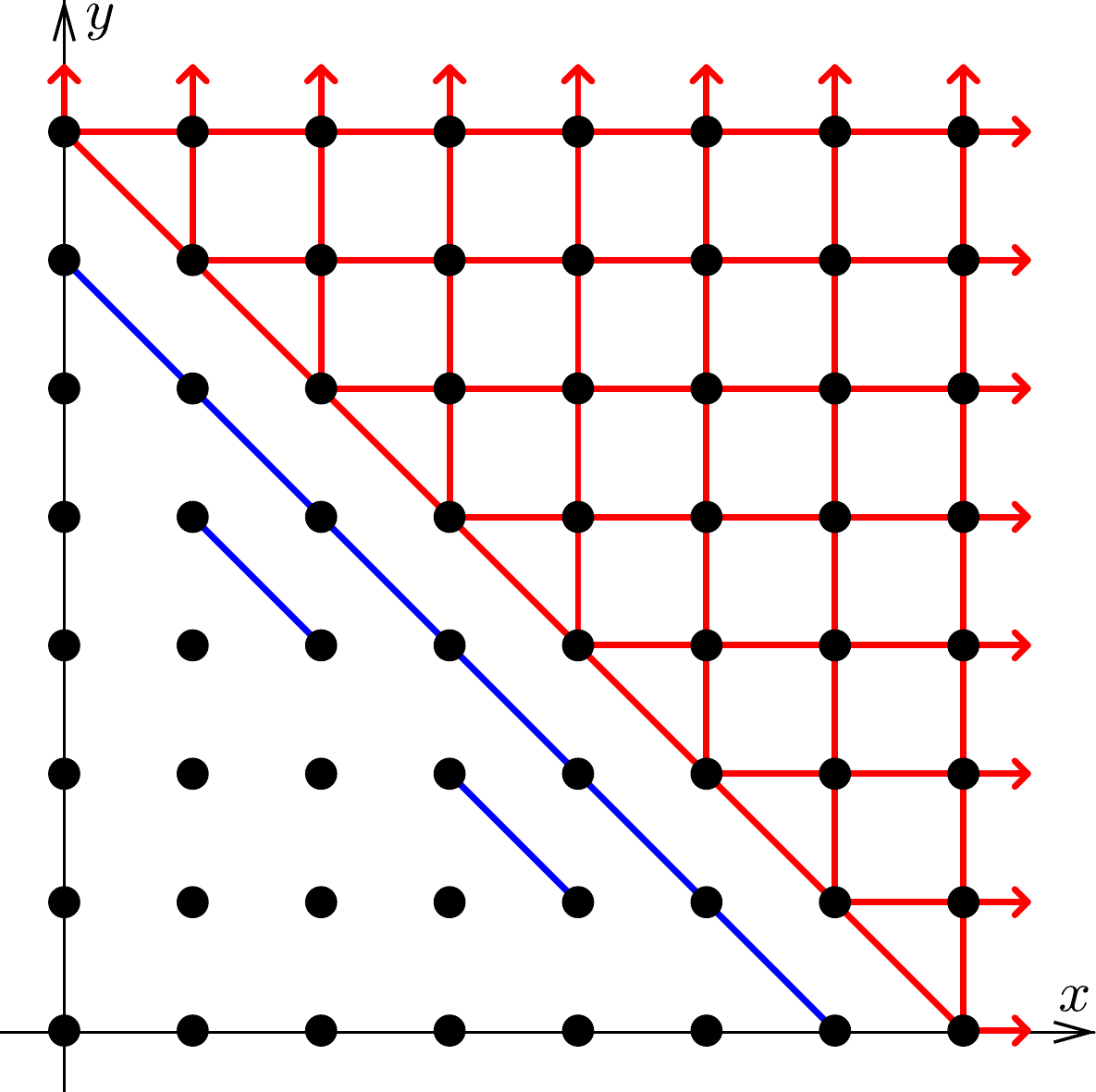}
  \\
  I &  & J
  \end{array}
$$
\caption[Two binomial ideals with soccular components containing
non-binomial socle elements]{The congruences induced by $I = \<x^2y -
xy^2, x^4 - x^3y, xy^3 - y^4, x^5\>$ (left) and $J = \<x^4y - x^3y^2,
x^2y^3 - xy^4, x^6 - x^5y, xy^5 - y^6, x^7\>$ (right) on $\NN^2$.  $I$
admits a binomial irreducible decomposition, but $J$ does not.}%
\label{fig:ex304}
\end{figure}
The ideal $I$ has three key witnesses aside from its cogenerator, and
the binoccular decomposition produced in
Theorem~\ref{t:binocculardecomp} has a component at each of these key
witnesses.  Any one of these three can be omitted, and omitting the
component cogenerated by $x^2y$ yields a binomial irreducible
decomposition of~$I$.  In contrast, $J$ has four non-maximal key
witnesses, two of which cogenerate binoccular components that fail to
admit binomial irreducible decompositions.  Since only one can be
omitted, $J$ does not admit a binomial irreducible decomposition.
\end{example}

\begin{prob}\label{p:futbinirrconstruct}
Determine when all of the binoccular components without simple socle
can be omitted from the decomposition in
Theorem~\ref{t:binocculardecomp}.
\end{prob}

\begin{question}\label{q:futbinirrexist}
Which binomial ideals admit binomial irreducible decompositions?  
\end{question}

Question~\ref{q:futbinirrexist} is more general than
Problem~\ref{p:futbinirrconstruct} but may involve primary
decompositions that do not arise from mesoprimary decomposition.

\section{Irreducible decomposition of binomial ideals}\label{s:irredecomp}

This section produces an irreducible decomposition of any given
binomial ideal.  We first define the irreducible closure of a
coprincipal binomial ideal (Definition~\ref{d:irreducibleclosure}).
Unlike a binoccular closure (Definition~\ref{d:binoccularclosure}),
which may have non-binomial elements in its socle, the cogenerators of
coprincipal binomial ideals are the only socle elements that survive
irreducible closure.

\begin{defn}\label{d:irreducibleclosure}
For a $P$-coprincipal binomial ideal $I \subset \kk[Q]$ cogenerated by
$w \in Q$,~set $R_P = \kk[Q_P]/I_P$ and let $G_P \subset Q_P$ denote
the group of units.  Write $\ow^\perp$ for the unique graded
$\kk$-vector subspace of~$R_P$ such that
$R_P = (\kk[G_P] \cdot \ttt^\ow) \oplus \ow^\perp$. Let
$\ow_\infty^\perp$ denote the largest $\kk[Q_P]$-submodule of $R_P$
that lies entirely in $\ow^\perp$ and set
$\oR_P = R_P/\ol w_\infty^\perp$.  The \emph{irreducible closure}
of~$I$ is the ideal $\irr(I) = \ker(\kk[Q] \to \oR_P)$.
\end{defn}

\begin{example}\label{e:irreducibleclosure}
Let $I = \<x^2y - xy^2, x^3, y^3, z^3\>$ and $\mm_P =
\<x,y,z\>$.  Then $z^2(x^2 + y^2 - xy)$ lies in $\soc_P(I)$ and thus
generates a $\kk[x,y,z]$-submodule of $(x^2yz^2)^\perp$.  On the other
hand, the element $z(x^2 + y^2 - xy)$ lies in $\soc_P(\<z^2(x^2 + y^2
- xy)\> + I)$ but outside of $\soc_P(I)$.  Continuing yields the
irreducible closure $\irr(I) = \<x^2 + y^2 - xy\> + I$ of~$I$.
\end{example}

Recall the usual notion of essentiality from commutative algebra: a
submodule $N$ of a module $M$ is \emph{essential} if $N$ intersects
every nonzero submodule of~$M$ nontrivially.

\begin{lemma}\label{l:essentialperp}
If $I \subset \kk[Q]$ is a $P$-coprincipal binomial ideal with
monomial cogenerator~$\ttt^w$, then $\<\ttt^\ow\> = \kk[G_P] \cdot
\ttt^\ow$ is an essential\/ $\kk[Q_P]$-submodule of~$\oR_P$ that is
isomorphic to a Gorenstein quotient of\/~$\kk[G_P]$.
\end{lemma}
\begin{proof}
The equality $\<\ttt^\ow\> = \kk[G_P] \cdot \ttt^\ow$ follows because
$\ttt^\ow$ is annihilated by $\mm_P$.  The Gorenstein condition
\cite[Appendix~A.7]{Vas} holds because the kernel of the surjection
$\kk[G_P] \to \<\ttt^\ow\>$ is, after faithfully flat extension to an
algebraically closed coefficient field, generated by a binomial
regular sequence \cite[Theorem~2.1(b)]{ES96}.  To prove essentiality,
first note that $\soc_P(\oR_P)$ is an essential submodule of~$\oR_P$
because $\mm_P$ is nilpotent on~$\oR_P$, and then note that
$\<\ttt^\ow\> = \soc_P(\oR_P)$ by construction of~$\oR_P$.
\end{proof}

\begin{prop}\label{p:assperp}
Fix a $P$-coprincipal binomial ideal $I \subset \kk[Q]$ with monomial
cogenerator $\ttt^w$.  The associated primes of~$R_P$, $\oR_P$, and
$R_P/\mm_P$ coincide and are all~minimal.
\end{prop}
\begin{proof}
The associated primes of~$R_P$ and $R_P/\mm_P$ coincide by
Proposition~\ref{p:kmprimarydecomp}.  But
$R_P/\mm_P \cong \<\ttt^\ow\> \subseteq R_P$ maps isomorphically to an
essential submodule of $\oR_P$ by Lemma~\ref{l:essentialperp}, so all
three sets of associated primes coincide.
\end{proof}

Compare the next result to the coprincipal special case of
Proposition~\ref{p:kmprimarydecomp}.

\begin{thm}\label{t:irrclosureprimarydecomp}
The irreducible closure $\irr(I)$ of any coprincipal ideal $I$ has a
unique minimal primary decomposition.  Every primary component therein
is irreducible.
\end{thm}
\begin{proof}
Minimality of all associated primes in Proposition~\ref{p:assperp}
implies the first statement.  Since localization preserves
essentiality \cite[Corollary~1.3]{Bass62}, the ordinary
localization~$\<\ttt^\ow\>_\pp$ at the prime $\kk[Q]$-ideal $\pp$ is
an essential submodule of~$(\oR_P)_\pp$ for every
$\pp \in \ass(\irr(I))$ by Lemma~\ref{l:essentialperp}.  The same
lemma implies that $\<\ttt^\ow\>_\pp$ is Gorenstein of dimension~$0$,
so $\<\ttt^\ow\>_\pp$ has simple socle.  Thus the quotient by
$\irr(I)_\pp$ has simple socle, whence $\irr(I)_\pp$ is irreducible by
Lemma~\ref{l:vasconcelos}.
\end{proof}

\begin{remark}\label{r:Gor}
The proof of Theorem~\ref{t:irrclosureprimarydecomp} via
Lemma~\ref{l:essentialperp} and Proposition~\ref{p:assperp} shows,
quite generally, that if a Noetherian ring is contained in a
localization that has an essential submodule isomorphic to a
Gorenstein ring, then the original ring has a unique minimal primary
decomposition all of whose components are quotients modulo irreducible
ideals.
\end{remark}

We now extend Theorem~\ref{t:binocculardecomp} to irreducible closures
before stating Corollary~\ref{c:irreducibledecomp}, our main result
for this section.

\begin{thm}\label{t:irreducibledecomp}
Every binomial ideal $I \subset \kk[Q]$ equals the intersection of the
irreducible closures of the coprincipal components cogenerated by its
essential witnesses.
\end{thm}
\begin{proof}
Fix a monoid prime $P \subset Q$ and nonzero $f \in \soc_P(I)$.  By
Definition~\ref{d:mesoprime}.4, some nonzero monomial $\lambda \ttt^w$
of $f$ is an essential $I_P$-witness for $P$.  Every monomial of~$f$
that is nonzero modulo $\irr(W_w^P(I))_P$ lies in the submodule
$\<\ttt^\ol w\>$ of~$\kk[Q]_P/\irr(W_w^P(I))_P$, so $f$ is nonzero
modulo $\irr(W_w^P(I))_P$.$\!$  Lemma~\ref{l:intersection} completes the
proof.
\end{proof}

\begin{cor}\label{c:irreducibledecomp}
Fix a binomial ideal $I \subset \kk[Q]$.  An irreducible decomposition
of~$I$ results by intersecting the canonical primary components of the
irreducible closures of the coprincipal components cogenerated by the
essential $I$-witnesses.
\end{cor}
\begin{proof}
Apply Theorem~\ref{t:irreducibledecomp}, then
Theorem~\ref{t:irrclosureprimarydecomp}.
\end{proof}


\end{document}